\documentclass[reqno]{amsart}

\usepackage{amssymb}
\usepackage{graphicx}
\usepackage{mathrsfs}
\usepackage{hyperref}
\usepackage{color}
\usepackage{url}
\usepackage{cite}

\usepackage{times}
\usepackage{microtype}
\usepackage{amssymb}
\usepackage{mathtools}
\usepackage{tikz}
\usepackage{wasysym}
\usepackage{centernot}
\usepackage{color}
\usepackage{setspace}
\usepackage{appendix}
\usepackage{booktabs}
\usepackage{multirow}

\def\bR{{\mathbb R}}

\def\sE{{\mathscr E}}
\def\sF{{\mathscr F}}

\def\sG{{\mathscr G}}

\def\re{{\mathrm{e}}}

\def\cH{\mathcal{H}}
\def\sL{\mathscr{L}}
\def\cD{\mathcal{D}}

\def\cH{\mathcal{H}}

\def\bS{\mathbf{S}}
\def\bs{\mathbf{s}}

\def\${|\!|\!|}
\def\l|{\left|\!\left|\!\left|}
\def\r|{\right|\!\right|\!\right|}

\newtheorem{theorem}{Theorem}[section]
\newtheorem{lemma}[theorem]{Lemma}

\newtheorem{corollary}[theorem]{Corollary}
\theoremstyle{definition}
\newtheorem{definition}[theorem]{Definition}

\theoremstyle{remark}
\newtheorem{remark}[theorem]{Remark}
\numberwithin{equation}{section}

\begin{document}

\title[On structure of harmonic functions]{On structure of $L^2$-harmonic functions for one-dimensional diffusions}

\author{Liping Li}
\address{Fudan University, Shanghai, China.  }
\address{Bielefeld University,  Bielefeld, Germany.}
\email{liliping@amss.ac.cn}
\thanks{The author is partially supported by NSFC (No.  11931004) and Alexander von Humboldt Foundation in Germany.}






\begin{abstract}
In this note we analyse the harmonic functions in $L^2$-sense for an irreducible diffusion on an interval. 
\end{abstract}

\maketitle

\tableofcontents

\section{Dirichlet forms associated with one-dimensional diffusions}

Let $I:=\left\langle l,r\right\rangle$ be an interval where $l$ or $r$ may or may not be in $I$ and $m$ be a positive Radon measure on $I$ with full topological support.  The notation $m(l+)<\infty$ (resp.  $m(r-)<\infty$) means that for some $\varepsilon>0$,  $m\left((l,l+\varepsilon)\right)<\infty$ (resp.  $m\left((r-\varepsilon, r)\right)<\infty$).  Otherwise write $m(l+)=\infty$ (resp.  $m(r-)=\infty$).  When $l=-\infty$ (resp.  $r=\infty$),  $l+\varepsilon$ (resp.  $r-\varepsilon$) in this meaning is replaced by some constant in $(l,r)$.  Clearly,  if $l\in I$ (resp.  $r\in I$),  then $m(l+)<\infty$ (resp.  $m(r-)<\infty$).  
 Fix a point $e\in \mathring I:=(l,r)$ and without loss of generality assume that $m(\{e\})=0$.   Denote by
\[
\bS(\mathring I):=\{\bs:\mathring I\rightarrow \bR: \bs\text{ is continuous and strictly increasing}, \; \bs(e)=0\}
\]
the family of all scale functions on $\mathring I$. 
Define
\[
	\bs(l):=\lim_{x\downarrow l}\bs(x)\geq -\infty,\quad \bs(r):=\lim_{x\uparrow r}\bs(x)\leq \infty.  
\]
Set $\bar{I}:=[l, r]$ to be the interval containing the boundary points even if $l$ or $r$ is infinite.  If a function $f$ is undefined in $l$ or $r$,  we understand $f(l)$ or $f(r)$ the limit $\lim_{x\rightarrow l\text{ or }r}f(x)$.  
 Take $\bs\in \bS(\mathring{I})$ and set for $x\in \bar{I}$,  
\[
	\sigma(x):=\int_e^x\int_e^\xi m(d\eta)d\bs(\xi),\quad \mu(x):=\int_e^x \int_e^\xi d\bs(\eta)m(d\xi).
\]
Here (and hereafter) $\int_e^\xi m(d\eta)$ means $\int_{(e,\xi]} m(d\eta)$ for $\xi>e$ and $-\int_{(\xi, e]}m(d\eta)$ for $\xi<e$,  and analogical meaning holds for $\int_e^x m(d\xi)$.  
The following classification of the boundary point is well known; see,  e.g.,  \cite[\S5.12]{I06}.  

\begin{definition}
Given $\bs\in \bS(\mathring I)$,  $l$ (resp.  $r$) is called (with respect to $(\bs,m)$)
\begin{itemize}
\item[(1)] \emph{regular},  if $\sigma(l),\mu(l)<\infty$ (resp.  $\sigma(r), \mu(r)<\infty)$);
\item[(2)] \emph{exit},  if $\sigma(l)<\infty, \mu(l)=\infty$ (resp. $\sigma(r)<\infty, \mu(r)=\infty$);
\item[(3)] \emph{entrance},  if $\sigma(l)=\infty,  \mu(l)<\infty$ (resp.  $\sigma(r)=\infty,  \mu(r)<\infty$);
\item[(4)]  \emph{natural},  if $\sigma(l)=\mu(l)=\infty$ (resp. $\sigma(r)=\mu(r)=\infty$).  
\end{itemize}
In addition,  $l$ (resp.  $r$) is called \emph{absorbing},  if $l$ (resp.  $r$) is regular and $l\notin I$ (resp.  $r\notin I$).  It is called \emph{reflecting},  if  it is regular and contained in $I$. 
\end{definition}
\begin{remark}
Note that $r$ is regular,  if and only if $\bs(r)<\infty$ and $m(r-)<\infty$.  If $r$ is exit,  then $\bs(r)<\infty$ and $m(r-)=\infty$.  If $r$ is entrance,  then $\bs(r)=\infty$ and $m(r-)<\infty$.  If $r$ is natural,  then at least one of $m(r-)$ and $\bs(r)$ must be $\infty$. 
\end{remark}

The endpoint $l$ or $r$ is called \emph{approachable} if $\bs(l)>-\infty$ or $\bs(r)<\infty$.  
Given a function $f$ on $I$,  $f\ll d\bs$ means that $f$ is absolutely continuous with respect to $d\bs$,  i.e.  there exists an absolutely continuous function $g$ on $\bs(I)=\{\bs(x): x\in I\}$ such that $f=g\circ \bs$.  Meanwhile $df/d\bs:=g'\circ \bs$.  Note that if $l$ or $r$ is approachable,  then any function $f$ with $f\ll d\bs$ and $df/d\bs\in L^2(I,d\bs)$ admits the finite limit $f(l):=\lim_{x\downarrow l}f(x)$ or $f(r):=\lim_{x\uparrow r} f(x)$.  Particularly  $f\in C((l,r])$ or $f\in C([l,r))$;  see \cite[\S2.2.3]{CF12}.  

What we are concerned with is a regular and irreducible Dirichlet form $(\sE,\sF)$ on $L^2(I,m)$.  To state its representation in the following lemma,  denote by $\tilde{\bS}(\mathring I)$ the family of all scale functions $\bs\in \bS(\mathring I)$ satisfying the condition: For $j=l$ or $r$,  if $j\in I$ and $m(\{j\})>0$,  then $|\bs(j)|<\infty$.  

\begin{lemma}
Let $I,m$ be given as above.  Then $(\sE,\sF)$ is a regular,  irreducible and strongly local Dirichlet form on $L^2(I,m)$,  if and only if  there exists a unique scale function $\bs\in \tilde\bS(\mathring{I})$ such that 
\begin{equation}\label{eq:12}
\begin{aligned}
	&\sF=\{f\in L^2(I,m): f\ll d\bs,  df/d\bs\in L^2(I,d\bs), \\
	&\qquad\qquad \qquad \qquad f(j)=0\text{ if }j \text{ is absorbing for }j=l\text{ or }r\},\\
	&\sE(f,g)=\frac{1}{2}\int_I \frac{df}{d\bs}\frac{dg}{d\bs}d\bs,\quad f,g\in \sF.  
\end{aligned}
\end{equation}
\end{lemma}
\begin{proof}
\emph{Necessity}.  Let $(\sE,\sF)$ be such a Dirichlet form.  Applying \cite[Theorem~2.1]{LY19} and the irreducibility,  we can obtain a unique interval $J\subset I$,  called effective interval,  and a unique adapted scale function on $J$ in the sense of \cite[(2.2)]{LY19} representing $(\sE,\sF)$.  Note that $J$ must be ended by $l$ and $r$,  because otherwise $I\setminus J$ would become a non-trivial $m$-invariant set (with respect to $(\sE,\sF)$) as violates the irreducibility.  Denote by the adapted scale function on $J$ by $\bs$.  Then $\bs\in \bS(\mathring I)$.  To show $\bs\in \tilde{\bS}(\mathring I)$,  argue by contradiction and suppose that $r\in I$,  $m(\{r\})>0$ and $\bs(r)=\infty$.  The adaptedness of $\bs$ indicates that $r\notin J$.  Hence $\{r\}$ is a non-trivial $m$-invariant set,  as leads to a contradiction.  Therefore $\bs\in \tilde{\bS}(\mathring{I})$ and the expression of $(\sE,\sF)$ can be obtained by \cite[Theorem~2.1]{LY19}. 

\emph{Sufficiency}.  Let $J:=\left\langle l, r\right \rangle$,  where $j\in J$ if and only if $j$ is reflecting with respect to $(\bs, m)$ for $j=l$ or $r$.  Then $J\subset I$ and $\bs$ is adapted to $J$.  By \cite[Theorem~2.1]{LY19},  \eqref{eq:12} gives a regular and strongly local Dirichlet form on $L^2(I,m)$ with the effective interval $J$ and adapted scale function $\bs$.  The condition $\bs\in \tilde{\bS}(\mathring {I})$ implies that $m(I\setminus J)=0$ and hence $(\sE,\sF)$ is irreducible.  That completes the proof. 
\end{proof}
\begin{remark}
The associated Markov process $X=(X_t)_{t\geq 0}$ of $(\sE,\sF)$ is a diffusion process on $I$ with no killing inside  whose scale function is $\bs$ and speed measure is $m$; see,  e.g.,  \cite[V\S7]{RW87}.
\end{remark}



From now on we denote by $I_e$ the effective interval of $(\sE,\sF)$ as described in \cite[\S2.3]{LY19}.  More precisely,  $\mathring I\subset I_e\subset I$,  and $j\in I_e$ if and only if $j$ is reflecting for $j=l$ or $r$.  In addition,  $j\in I\setminus I_e$ implies that $|\bs(j)|=\infty$ and $m(\{j\})=0$.  Any singleton contained in $I_e$ is of positive capacity and $I\setminus I_e$ is $\sE$-polar.  Particularly,  $\sF\subset C(I_e)$.  The family
\begin{equation}\label{eq:11}
\mathscr C_{I_e}:=\{\varphi\circ \bs: \varphi\in C_c^\infty(\bs(I_e))\}
\end{equation}
is a special standard core of $(\sE,\sF)$.  

\subsection*{Notations}
Given an interval $J$,  $C(J)$, $pC(J)$ and $C_c^\infty(J)$ stand for the families of all continuous functions,  all non-negative continuous functions and all smooth functions with compact support on $J$ respectively. 

\section{Solutions of harmonic equation}

Given a constant $\alpha> 0$,  consider the following equation
\begin{equation}\label{eq:21}
	\frac{1}{2}\frac{d}{dm}\frac{du}{d\bs}(x)=\alpha u(x),\quad x\in \mathring{I}.  
\end{equation}
A solution of \eqref{eq:21} means $u\in L^1_\text{loc}(\mathring{I},m)$ such that $u\ll d\bs$ and a $d\bs$-a.e. version $v$ of $du/d\bs$ satisfies 
\begin{equation}\label{eq:22}
	2\alpha\int_{(x,y]} u(\xi)m(d\xi)=v(y)-v(x),\quad \forall x,y\in \mathring{I},  x<y.  
\end{equation}
Note that \eqref{eq:22} indicates that $v$ is right continuous and $v(r):=\lim_{x\uparrow r}v(x)$ (resp.  $v(l):=\lim_{x\downarrow l}v(x)$) is well defined if $u|_{[e,r)}\in L^1([e,r),m)$ (resp.  $u|_{(l,e]}\in L^1((l,e],m)$).

This section is devoted to presenting two particular solutions of \eqref{eq:21} as obtained in \cite[Chapter II]{M68}; see also \cite[\S5.12]{I06}.  
Set,  for $n=0,1,2,\cdots,  x\in \mathring{I}$,  
\[
 	u^0(x)\equiv 1,\quad u^1(x)=\sigma(x),\quad  u^{n+1}(x)=\int_e^x\int_e^y u^n(\xi)m(d\xi)d\bs(y). 
\]
Denote
\[
	u(x)=\sum_{n=0}^\infty \alpha^n u^n(x),\quad x\in \mathring I,
\]
and introduce the functions 
\[
\begin{aligned}
	u_+(x)&=u(x)\int_x^r u(y)^{-2}d\bs(y),\quad x\in \mathring I\\
	u_-(x)&=u(x)\int_l^x u(y)^{-2}d\bs(y),\quad x\in \mathring I.  
\end{aligned}
\]
The following lemma due to \cite{M68} is crucial to our treatment. 

\begin{lemma}\label{LM21}
\begin{itemize}
\item[(1)] $1+\sigma(x)\alpha\leq u(x)\leq \exp\{\alpha \sigma(x)\}$.   For $x>e$ (resp. $x<e$),
\[
\begin{aligned}
\int_x^r u(y)^{-2}d\bs(y)&\leq \frac{1}{\alpha (1+\alpha \sigma(x))m((e,x])},\quad \\
& \left(\text{resp. }  \int_l^x u(y)^{-2}d\bs(y)\leq \frac{1}{\alpha (1+\alpha \sigma(x))m((x,e])} \right). 
\end{aligned}\]
\item[(2)] $u,  u_\pm \in pC(\mathring I)$ are solutions of \eqref{eq:21}. 
\item[(3)] $u_+$ is decreasing.  If $r$ is not entrance,  then $u_+(r)=0$.  $du_+/d\bs$ is increasing and if $r$ is entrance or natural,  then $du_+/d\bs(r)=0$.  
\item[(4)] $u_-$ is increasing in $x$.  If $l$ is not entrance,  then $u_-(l)=0$.  $du_-/d\bs$ is increasing in $x$ and if $l$ is entrance or natural,  then $du_-/d\bs(l)=0$.  
\end{itemize}
\end{lemma}
\begin{proof}
See \cite[II,  \S2]{M68}. 
\end{proof}

\section{Analytic treatment of harmonic functions}

Let $(\sE^0,\sF^0)$ be the part Dirichlet form of $(\sE,\sF)$ on $\mathring{I}$,  i.e.  
\[
\begin{aligned}
	&\sF^0:=\{f\in \sF: f=0\text{ q.e. on }I\setminus \mathring{I}\}, \\
	&\sE^0(f,g):=\sE(f,g),\quad f,g\in \sF^0.  
\end{aligned}\]
It is a regular Dirichlet form on $L^2(\mathring I, m)$ having a special standard core
\[
\mathscr C_{\mathring I}:=\{\varphi\circ \bs: \varphi\in C_c^\infty(\bs(\mathring I))\}.  
\]
For any $\alpha>0$,  $\sF^0$ is a closed subspace of $\sF$ under the norm $\|\cdot\|_{\sE_\alpha}$ and the following direct product decomposition holds true:
\[
	\sF=\sF^0\oplus_{\sE_\alpha} \cH_\alpha,
\]
where $\cH_\alpha:=\{f\in \sF: \sE_\alpha(f,g)=0, \forall g\in \sF^0\}$.  The function in $\cH_\alpha$ is called $\alpha$-harmonic.  

We note that $\sF^0=\sF$ if and only if $I\setminus \mathring{I}$ is $\sE$-polar,  i.e.  neither $l$ nor $r$ is reflecting,  or equivalently,  $I_e=\mathring{I}$.  In this case $\cH_\alpha=\{0\}$ for all $\alpha>0$.  When $r$ is reflecting (resp.  $l$ is reflecting),  Lemma~\ref{LM21}~(1) yields that $u_-(r)\in (0,\infty)$ (resp.  $u_+(l)\in (0,\infty)$).  Meanwhile set
\begin{equation}\label{eq:31}
	u^\alpha_r(x):=\frac{u_-(x)}{u_-(r)},\quad x\in (l,r];  \quad \left(\text{resp.  }u^\alpha_l(x):=\frac{u_+(x)}{u_+(l)},\quad x\in [l,r)\right)
\end{equation}
It is worth pointing out that $u_+$ and $u_-$ are linear independent; see \cite[II,  \S3\#5]{M68}. 
The result below characterizes the $\alpha$-harmonic functions.  

\begin{theorem}\label{THM31}
For $\alpha>0$,  the following holds:
\[
	\cH_\alpha=\left\lbrace 
	\begin{aligned}
	&\{0\},\qquad\qquad\quad \text{when }  I_e=(l,r);\\
	&\text{span}\{u^\alpha_r\},\qquad\;\, \text{when }  I_e=(l,r]; \\
	&\text{span}\{u^\alpha_l\},\qquad\;\, \text{when }I_e=[l,r);\\
	&\text{span}\{u^\alpha_l,  u^\alpha_r\},\quad \text{when }I_e=[l,r].
	\end{aligned}
	\right. 
\]
\end{theorem}
\begin{proof}
Only the case $I_e=(l,r]$ will be treated,  and the other cases can be proved by a similar way.   

We first show that $u^\alpha_r\in \sF$.  To do this,  take $\varepsilon>0$.  On account of Lemma~\ref{LM21}~(2),  $u^\alpha_r$ is a solution of \eqref{eq:21}.  Denote the $d\bs$-a.e. version of $du^\alpha_r/d\bs$ satisfying \eqref{eq:22} still by $du^\alpha_r/d\bs$.  Both $u^\alpha_r$ and $du^\alpha_r/d\bs$ are of bounded variation on $(l+\varepsilon,r-\varepsilon)$,  and $u^\alpha_r$ is continuous.  Then it follows from \eqref{eq:22} that
\[
\begin{aligned}
	\int_{l+\varepsilon}^{r-\varepsilon} \left(\frac{du^\alpha_r}{d\bs}\right)^2d\bs&=\int_{l+\varepsilon}^{r-\varepsilon} \frac{du^\alpha_r}{d\bs}du^\alpha_r \\
	&=\frac{du^\alpha_r}{d\bs}\cdot u^\alpha_r \bigg |_{l+\varepsilon}^{r-\varepsilon}-2\alpha \int_{(l+\varepsilon, r-\varepsilon]}u^\alpha_r(\xi)^2m(d\xi)
\end{aligned}\]
Hence 
\begin{equation}\label{eq:23}
	\int_{l+\varepsilon}^{r-\varepsilon} \left(\frac{du^\alpha_r}{d\bs}\right)^2d\bs+2\alpha \int_{(l+\varepsilon, r-\varepsilon]}u^\alpha_r(\xi)^2m(d\xi)=\frac{du^\alpha_r}{d\bs}\cdot u^\alpha_r \bigg |_{l+\varepsilon}^{r-\varepsilon}.  
\end{equation}
Since $u^\alpha_r$ is positive and increasing and $u^\alpha_r(r)=1$,  it follows that $u^\alpha_r(l):=\lim_{x\downarrow l}u^\alpha_r(x)$ is finite.  Note that $du^\alpha_r/d\bs\geq 0$ since $u^\alpha_r$ is increasing.  In addition, \eqref{eq:22} yields that $d u^\alpha_r/d\bs$ is increasing and $du^\alpha_r/d\bs(r):=\lim_{y\uparrow r}du^\alpha_r/d\bs(y)$ is finite.  Letting $\varepsilon\downarrow 0$ in \eqref{eq:23} and noticing that $u^\alpha_r(l)=0$ if $l$ is absorbing in view of Lemma~\ref{LM21}~(4),  we get that $u^\alpha_r\in \sF$. 

Next we assert that $u^\alpha_r\in \cH_\alpha$.  Let $f\in \mathscr C_{\mathring I}=\{\varphi \circ \bs: \varphi\in C_c^\infty(\bs(\mathring I))\}$.  Mimicking the derivation of \eqref{eq:23},  one can obtain that
\[
	\int_{l+\varepsilon}^{r-\varepsilon} \frac{du^\alpha_r}{d\bs}\frac{df}{d\bs}d\bs+2\alpha \int_{(l+\varepsilon, r-\varepsilon]}u^\alpha_r(\xi)f(\xi)m(d\xi)=\frac{du^\alpha_r}{d\bs}\cdot f \bigg |_{l+\varepsilon}^{r-\varepsilon}.  
\]
Note that $\lim_{\varepsilon\downarrow 0} f(r-\varepsilon)=\lim_{\varepsilon\downarrow 0}f(l+\varepsilon)=0$.  Hence $\sE_\alpha(u^\alpha_r,f)=0$.  Since $\mathscr C_{\mathring I}$ is a special standard core of $(\sE^0,\sF^0)$,  we get $u^\alpha_r\in \cH_\alpha$.  

Thirdly we prove that $\cH_\alpha\subset \text{span}\{u_+,u_-\}$.  Let $h\in \cH_\alpha$. Take an arbitrary interval $J:=(a,b)\subset [a,b]\subset I_e$ with $e\in J$.  For any $f\in \sF$ with $\text{supp}[f]\subset J$,  it holds $\sE_\alpha(h,f)=0$.  Set 
\[
	F_h(x):=\int_e^x h(\xi)m(d\xi).  
\]
Then $F_h$ is of bounded variation on $[a,b]$ and $f\in C([a,b])$ is also of bounded variation.  Hence
\[
	\int_I f(x)h(x)m(dx)=\int_{[a,b]} f(x)dF_h(x)=-\int_{[a,b]}F_h(x)df(x).
\]
It follows from $\sE_\alpha(h, f)=0$ that 
\[
	\int_{[a,b]}\left(\frac{dh}{d\bs}(x)-2\alpha F_h(x)\right)df(x)=0.  
\]
Note that $\mathscr{C}_{J}:=\{\varphi\circ \bs: \varphi\in C_c^\infty(\bs(J))\}\subset \sF$ and $\text{supp}[f]\subset J$ for any $f\in \mathscr{C}_J$.   This yields that $dh/d\bs-2\alpha F_h$ is constant on $J$.  Particularly,  $h|_{[a,b]}\in C([a,b])$ is a solution of the equation \eqref{eq:21} restricted to $J$.  In view of \cite[II,  \S4]{M68},  one  can conclude that $h\in \text{span}\{u_+,u_-\}$.  

Finally it suffices to prove $u_+\notin \sF$.  In fact,  it is straightforward to verify that $du_+/d\bs(r)\cdot u_+(r)$ is finite.  Since $u_+\geq 0$ satisfies  \eqref{eq:22} and $du/d\bs(e)=0$, it follows that $du_+/d\bs$ is increasing and
\[
	\frac{du_+}{d\bs}(l)\leq \frac{du_+}{d\bs}(e)=-\frac{1}{u(e)}=-1.  
\]
If $\sigma(l)=\infty$,  then Lemma~\ref{LM21}~(1) yields that $u_+(l)=\infty$.  An analogical derivation of \eqref{eq:23} leads to $u_+\notin \sF$.  If $\sigma(l)<\infty$,  then Lemma~\ref{LM21}~(1) implies that
\[
	u_+(l)\geq \int_l^r u(y)^{-2}d\bs(y)>0.  
\]
Meanwhile $m(l+)=\infty$ leads to $u_+\notin L^2(I,m)$ and $m(l+)<\infty$   corresponds to an absorbing endpoint $l$.  In the latter case $u_+\notin \sF$ because $u_+(l)\neq 0$.  That completes the proof.  
\end{proof}
\begin{remark}
The third and final steps in this proof would be simplified if we apply the probabilistic representation of harmonic functions \eqref{eq:41}.  In the case $I_e=(l,r]$,  \eqref{eq:41} shows that $\cH_\alpha$ is of dimension $1$.  Then $\cH_\alpha=\text{span}\{u^\alpha_r\}$ can be concluded after $u_r^\alpha\in \cH_\alpha$ is obtained. 
\end{remark}

Let us turn to consider the special case $\alpha=0$.  Denote by $\sF_\re$ the extended Dirichlet space of $(\sE,\sF)$,  and 
\[
	\sF^0_\re:=\{f\in \sF_\re: f=0,\text{ q.e. on }I\setminus \mathring{I}\}.  
\]
Define $\cH:=\{u\in \sF: \sE(u, f)=0,\forall f\in \sF^0_\re\}$.  Every function in $\cH$ is called harmonic.  Before moving on,  we prepare a lemma to give the expression of $\sF_\re$.  

\begin{lemma}\label{LM33}
It holds that
\[
\begin{aligned}
\sF_\re=\{f<\infty,  m\text{-a.e.}: &f\ll d\bs,  df/d\bs\in L^2(I,d\bs), \\
 & f(j)=0 \text{ if }|\bs(j)|<\infty\text{ but }j\notin I \text{ for }j=l\text{ or }r\}.  
\end{aligned}\]
\end{lemma}
\begin{proof}
See \cite[Theorem~2.2]{F14}.  
\end{proof}

When $|\bs(l)|+|\bs(r)|<\infty$,  set
\[
	u^0_l(x):=\frac{\bs(r)-\bs(x)}{\bs(r)-\bs(l)},\quad u^0_r(x):=\frac{\bs(x)-\bs(l)}{\bs(r)-\bs(l)},\quad x\in \mathring I.  
\]	
The main result characterizing harmonic functions is as follows. 

\begin{theorem}\label{THM34}
The following hold:
\begin{itemize}
\item[(1)]  When $I_e=(l,r)$,  $\cH=\{0\}$ if $\bs(l)>-\infty$ or $\bs(r)<\infty$; otherwise $\cH=\text{span}\{1\}$.  
\item[(2)] When $I_e=[l,r)$,  $\cH=\text{span}\{u^0_l\}$ if $\bs(r)<\infty$; otherwise $\cH=\text{span}\{1\}$. 
\item[(3)] When $I_e=(l,r]$,  $\cH=\text{span}\{u^0_r\}$ if $\bs(l)>-\infty$; otherwise $\cH=\text{span}\{1\}$. 
\item[(4)] When $I_e=[l,r]$,  $\cH=\text{span}\{u^0_l,u^0_r\}$. 
\end{itemize}
\end{theorem}
\begin{proof}
\begin{itemize}
\item[(1)] In this case $\sF_\re^0=\sF_\re$.  Note that $(\sE,\sF)$ is transient,  if and only if $\bs(l)>-\infty$ or $\bs(r)<\infty$;  see,  e.g., \cite[Theorem~2.2.11 and Example~3.5.7]{CF12}.  Under the transience,  $\sE(f,f)=0$ for $f\in\sF_\re$ implies $f=0$.  Hence $\cH=\{0\}$.  When $(\sE,\sF)$ is recurrent,  it follows from \cite[Theorem~5.2.16]{CF12} that $f$ is constant for any $f\in \sF_\re$ with $\sE(f,f)=0$.  To the contrary,  the recurrence of $(\sE,\sF)$ implies that $1\in \sF_\re$ and $\sE(1,f)=0$ for any $f\in \sF_\re$.  These yield $\cH=\text{span}\{1\}$. 
\item[(2)] It is straightforward to verify that $u^0_l\in \sF_\re,  \sE(u^0_l, f)=0$ for any $f\in \sF_\re^0$ if $\bs(r)<\infty$,  and $1\in \sF_\re,  \sE(1,f)=0$ for any $f\in \sF_\re^0$ if $\bs(r)=\infty$.  Hence $\text{span}\{u^0_l\}\subset \cH$ if $\bs(r)<\infty$ and $\text{span}\{1\}\subset \cH$ if $\bs(r)=\infty$.  To the contrary,  take $h\in \cH$.  For any $f\in \sF_\re^0$,  $\sE(h,f)=0$ implies that
\begin{equation}\label{eq:33}
	\int_l^r \frac{dh}{d\bs}df=0.  
\end{equation}
Since  $\mathscr C_{\mathring I}=\{\varphi\circ \bs: \varphi\in C_c^\infty(\bs(\mathring{I}))\}\subset \sF^0_\re$,  it follows from \eqref{eq:33} that $dh/d\bs$ is constant,  i.e.  $h=c_1\cdot \bs+c_2$ for two constants $c_1$ and $c_2$.  Consequently Lemma~\ref{LM33} yields that $h=c\cdot u^0_l$ for some constant $c$ if $\bs(r)<\infty$ and $h$ is constant if $\bs(r)=\infty$.  
\item[(3)] This case can be treated by an analogical way to (2). 
\item[(4)] It is easy to verify that $u^0_l,  u^0_r\in \sF_\re$ and $\sE(u^0_l,f)=\sE(u^0_r,f)=0$ for any $f\in \sF^0_\re$.  Thus $\text{span}\{u^0_l,u^0_r\}\subset \cH$.  To the contrary,  take $h\in \cH$ and \eqref{eq:33} indicates that $h=c_1\bs+c_2$ for some constants $c_1,c_2$.  Then there exist another two constants $\tilde{c}_1$ and $\tilde{c}_2$ such that $h=\tilde{c}_1u^0_l+\tilde{c}_2u^0_r$.  Therefore $\cH\subset \text{span}\{u^0_l,u^0_r\}$. 
\end{itemize}
That completes the proof. 
\end{proof}

\section{Probabilistic counterparts of harmonic functions}

Recall that $X$ is the diffusion process associated with $(\sE,\sF)$.  Denote by $\zeta$ the lifetime of $X$ and set
\[
	\tau:=\inf\{t>0:X_t\notin (l,r)\}.  
\]
Then $X_\tau\in \{l,r\}$ for $\tau<\zeta$.  For convenience,  write $\cH_0:=\cH$ and set $u_l^0:=1$ (resp.  $u_r^0:=1$) if $I_e=[l,r)$ and $\bs(r)=\infty$ (resp. $I_e=(l,r]$ and $\bs(l)=-\infty$).  
 The harmonic functions appearing in the previous section admit the following probabilistic representation.

\begin{theorem} Assume that $I_e\neq \mathring{I}$.  Then for any $\alpha\geq 0$ and $x\in I_e$,  
\[
\begin{aligned}
	u^\alpha_l(x)&=\mathbf{E}_x\left[\re^{-\alpha \tau}; X_\tau=l,  \tau<\zeta \right],\quad \text{if }l\in I_e;\\
	u^\alpha_r(x)&=\mathbf{E}_x\left[\re^{-\alpha \tau}; X_\tau=r,  \tau<\zeta \right],\quad \text{if }r\in I_e.
\end{aligned}\]
\end{theorem}
\begin{proof}
Set $\varphi^\alpha_j(\cdot):=\mathbf{E}_\cdot \left[\re^{-\alpha \tau}; X_\tau=j, \tau<\zeta \right]$ for $j=l$ or $r$.  
Suppose first that $\alpha>0$, or  that $\alpha=0$ and $(\sE,\sF)$ is transient.  Then
\begin{equation}\label{eq:41}
	\cH_\alpha=\left\{\mathbf{E}_\cdot \left[\re^{-\alpha \tau}f(X_\tau);\tau<\zeta \right]: f\in \sF\text{ for }\alpha>0\text{ and }f\in \sF_\re\text{ for }\alpha=0 \right\};
\end{equation}
see \cite[Theorems~3.2.2 and 3.4.2]{CF12}.  When $I_e=[l,r)$,  $X_\tau=l$ for $\tau<\zeta$.  Hence 
\[
	\cH_\alpha=\text{span}\{\varphi^\alpha_l\}.
\]
It follows from Theorems~\ref{THM31} and \ref{THM34} that $\varphi^\alpha_l=cu^\alpha_l$ for some constant $c$.  Note that $\varphi^\alpha_l(l)=1$ because $\mathbf{P}_l(\tau=0)=1$,  and $u^\alpha_l(l)=1$.  Therefore $c=1$.  Another case $I_e=(l,r]$ can be treated analogically and we can obtain that $u^\alpha_r=\varphi^\alpha_r$.  When $I_e=[l,r]$,  $\zeta=\infty$ and $X_\tau\in \{l,r\}$.  Then $\cH_\alpha=\text{span}\{\varphi^\alpha_l,  \varphi^\alpha_r\}$.  Meanwhile $\varphi^\alpha_l(l)=\varphi^\alpha_r(r)=1$ and $\varphi^\alpha_l(r)=\varphi^\alpha_r(l)=0$.  Therefore Theorem~\ref{THM31} yields $u^\alpha_l=\varphi^\alpha_l$ and $u^\alpha_r=\varphi^\alpha_r$.  

Finally consider $\alpha=0$ and $(\sE,\sF)$ is recurrent.  Particularly $\zeta=\infty$.  When $I_e=[l,r)$ or $I_e=(l,r]$,  it follows from \cite[Theorem~3.5.6~(2)]{CF12} that $\varphi^0_j(x)=\mathbf{P}_x(\sigma_j<\infty)=1$ for $j=l$ or $r$ and $x\in I_e$,  where $\sigma_j:=\inf\{t>0:X_t=j\}$. Hence $\varphi^0_j=u^0_j$.  When $I_e=[l,r]$,  \cite[Theorem~3.4.8]{CF12} implies that $\varphi^0_l,\varphi^0_r\in \cH$.  Since $\varphi^\alpha_l(l)=\varphi^\alpha_r(r)=1$ and $\varphi^\alpha_l(r)=\varphi^\alpha_r(l)=0$,  applying Theorem~\ref{THM34},  we can conclude that $u^0_l=\varphi^0_l,  u^0_r=\varphi^0_r$.  That completes the proof. 
\end{proof}

The special  case $\alpha=0$ leads to the characterization of the first hitting times of the endpoints $l$ and $r$.  

\begin{corollary}
Let $\sigma_j:=\inf\{t>0:X_t=j\}$ for $j=l$ or $r$.  Then the following hold:
\begin{itemize}
\item[(1)] When $I_e=[l,r)$,  $\mathbf{P}_x(\sigma_l<\infty)=1$ for any $x\in I_e$ if $\bs(r)=\infty$.  If $\bs(r)<\infty$,  then
\[
	\mathbf{P}_x(\sigma_l<\zeta)=\frac{\bs(r)-\bs(x)}{\bs(r)-\bs(l)},\quad x\in I_e. 
\]
\item[(2)] When $I_e=(l,r]$,  $\mathbf{P}_x(\sigma_r<\infty)=1$ for any $x\in I_e$ if $\bs(l)=-\infty$.  If $\bs(l)>-\infty$,  then
\[
	\mathbf{P}_x(\sigma_r<\zeta)=\frac{\bs(x)-\bs(l)}{\bs(r)-\bs(l)},\quad x\in I_e. 
\]
\item[(3)] When $I_e=[l,r]$,  
\[
	\mathbf{P}_x(\sigma_l<\sigma_r)=\frac{\bs(r)-\bs(x)}{\bs(r)-\bs(l)},\quad x\in I_e. 
\]
\end{itemize}
\end{corollary}

\section{Harmonic functions in generator domain}

Denote by $\sL$ with domain $\cD(\sL)$ the $L^2$-generator of $(\sE,\sF)$.  In this section we answer the question that whether an $\alpha$-harmonic function belongs to $\cD(\sL)$ for $\alpha>0$.  Given a function $f\in \sF$,  $\frac{d}{dm}(df/d\bs)\in L^2(\mathring I,m)$ means that there exist a $d\bs$-a.e. version of $df/d\bs$,  still denoted by $df/d\bs$,  and $g\in L^2(\mathring I,m)$ such that
\begin{equation}\label{eq:51}
	\frac{df}{d\bs}(x)-\frac{df}{d\bs}(y)=\int_{(x,y]}g(\xi)m(d\xi),\quad l<x<y<r.  
\end{equation}
Meanwhile $\frac{d}{dm}(df/d\bs):=g$.  
Note that if $j\in I_e$ for $j=l$ or $r$,  then \eqref{eq:51} implies that $df/d\bs(j):=\lim_{x\rightarrow j}df/d\bs(x)$ exists and is finite.  
The description of $\sL$ is given as follows.  A related consideration is referred to \cite{F14}. 

\begin{lemma}
The $L^2$-generator of $(\sE,\sF)$ is
\[
	\begin{aligned}
		&\cD(\sL)=\bigg\{f\in \sF: \frac{d}{dm}\frac{df}{d\bs}\in L^2(\mathring I,m),  \frac{df}{d\bs}(j)=0\text{ if }j\in I_e,  \\
		&\qquad \qquad \qquad \qquad \qquad \qquad \qquad \qquad \qquad m(\{j\})=0\text{ for }j=l\text{ or }r \bigg\}, \\
		&\sL f=\left\lbrace
		\begin{aligned}
		&\frac{1}{2}\frac{d}{dm}\frac{df}{d\bs}(x),\quad\quad  \quad \quad\;\; x\in \mathring{I},  \\
		&\frac{1}{2}\frac{df}{d\bs}(l)/m(\{l\}),\quad\quad\quad \text{if }l\in I_e, \\
		&-\frac{1}{2}\frac{df}{d\bs}(r)/m(\{r\}),\quad\;\; \text{if }r\in I_e,
		\end{aligned} \right.
	\end{aligned}
\]
where we make the convention $0/0:=0$.  
\end{lemma}
\begin{proof}
Let $\mathscr C_{I_e}$  be the special standard core of $(\sE,\sF)$ defined as \eqref{eq:11}.  Denote by $\sG$ the right term in the first identity.  Clearly $\sL f$ given by the second identity is well defined for any $f\in \sG$.  
It is straightforward to verify that $-\int_{I_e}\sL f g dm=\sE(f,g)$ for any $f\in \sG$ and $g\in \mathscr C_{I_e}$.  Hence we only need to show $\cD(\sL)\subset \sG$.  To do this,  take $f\in \cD(\sL)$ with $h:=\sL f\in L^2(I_e,m)$.  With out loss of generality assume that $df/d\bs(e)$ is well defined and finite. Then $\sE(f,g)=-\int_{I_e} h gdm$ for any $g\in \mathscr C_{I_e}$.   Define
\[
	F_h(x):=\int_{(e,x]}h(\xi)m(d\xi)+\frac{1}{2}\frac{df}{d\bs}(e),  \quad x>e
\]
and 
\[
	F_h(x):=-\int_{(x,e]}h(\xi)m(d\xi)+\frac{1}{2}\frac{df}{d\bs}(e),  \quad x<e.
\]
Set $F_h(l-):=-\int_{[l,e]}h(\xi)m(d\xi)+\frac{1}{2}\frac{df}{d\bs}(e)$ if $l\in I_e$.  For $j=l$ or $r$,  if $j\notin I_e$,  then $g(j)=0$.  Hence we have
\[
	-\int_{I_e}h g dm=-\int_{[l,r]}gdF_h=\int_l^r F_h(x)dg(x)+F_h(l-)g(l)-F_h(r)g(r),
\]
where we impose $F_h(l-)\cdot 0=F_h(r)\cdot 0=0$ even if $F_h(l-)$ or $F_h(r)$ is not defined.  It follows from $\sE(f,g)=-\int_{I_e}hgdm$ that
\begin{equation}\label{eq:52}
	\frac{1}{2}\int_l^r \frac{df}{d\bs}(x)dg(x)=\int_l^r F_h(x)dg(x)+F_h(l-)g(l)-F_h(r)g(r),\quad \forall g\in \mathscr C_{I_e}.  
\end{equation}
Taking all $g\in \mathscr C_{I_e}$ with $\text{supp}[g]\subset (l,r)$ and noticing $F_h(e)=\frac{1}{2}df/d\bs(e)$,  we can obtain that
\begin{equation}\label{eq:53}
	F_h(x)=\frac{1}{2}\frac{df}{d\bs}(x),\quad x\in (l,r). 
\end{equation}
Particularly $\frac{1}{2}\frac{d}{dm}\frac{df}{d\bs}=h\in L^2(\mathring{I},m)$.  In addition, \eqref{eq:52} and \eqref{eq:53} yield that $F_h(l-)=0$ if $l\in I_e$ and $F_h(r)=0$ if $r\in I_e$.  In the former case 
\[
	F_h(l-)=F_h(l)-h(l)m(\{l\})=\frac{1}{2}\frac{df}{d\bs}(l)-h(l)m(\{l\}).  
\]
Hence $df/d\bs(l)=0$ if $m(\{l\})=0$ and $h(l)=\frac{1}{2}\frac{df}{d\bs}(l)/m(\{l\})$.  The latter case can be treated analogically.  Therefore we can obtain $f\in \sG$.  That completes the proof. 
\end{proof}

Set $C:=\int_l^r u(y)^{-2}d\bs(y)$,  which is finite due to Lemma~\ref{LM21}~(1).  It is straightforward to calculate that if $l\in I_e$,  then
\[
\begin{aligned}
	&u_-(l)=0,\quad u_+(l)=Cu(l)\in (0,\infty), \\
	&\frac{du_-}{d\bs}(l)=\frac{1}{u(l)}\in (0,\infty),  \quad \frac{du_+}{d\bs}(l)=C\frac{du}{d\bs}(l)-\frac{1}{u(l)}\in (-\infty,0); 
\end{aligned}\]
and if $r\in I_e$,  then
\[
\begin{aligned}
&u_-(r)=Cu(r)\in (0,\infty),\quad u_+(r)=0, \\
&\frac{du_-}{d\bs}(r)=C\frac{du}{d\bs}(r)+\frac{1}{u(r)}\in (0,\infty),\quad \frac{du_+}{d\bs}(r)=-\frac{1}{u(r)}\in (-\infty,0).  
\end{aligned}
\]
Put $c_l:=-\frac{du_+}{d\bs}(l)/\frac{du_-}{d\bs}(l)$ and $c_r:=-\frac{du_+}{d\bs}(r)/\frac{du_-}{d\bs}(r)$.  Clearly $c_l, c_r\in (0,\infty)$. 

\begin{theorem}
\begin{itemize}
\item[(1)] When $I_e=[l,r)$,  $\cH_\alpha\cap \cD(\sL)=\{0\}$ if $m(\{l\})=0$ and $\cH_\alpha\cap \cD(\sL)=\cH_\alpha$ if $m(\{l\})>0$.
\item[(2)] When $I_e=(l,r]$,  $\cH_\alpha\cap \cD(\sL)=\{0\}$ if $m(\{r\})=0$ and  $\cH_\alpha\cap \cD(\sL)=\cH_\alpha$ if $m(\{r\})>0$.
\item[(3)] When $I_e=[l,r]$,  
\[
	\cH_\alpha \cap \cD(\sL)=\left\lbrace
	\begin{aligned}
		&\cH_\alpha,\qquad\qquad \qquad\quad m(\{l\}), m(\{r\})>0;  \\
		&\text{span}\{u_++c_r u_-\},\quad m(\{l\})>0,m(\{r\})=0;  \\
		&\text{span}\{u_++c_l u_-\},\quad\, m(\{r\})>0,m(\{l\})=0;  \\
		&\{0\},  \qquad\qquad \qquad\quad m(\{l\})=m(\{r\})=0. 
	\end{aligned}	
	 \right.
\]
\end{itemize}
\end{theorem}
\begin{proof}
For the first and second assertions,  it suffices to note that $du_+/d\bs(l)\neq 0$ if $l\in I_e$ and $du_-/d\bs(r)\neq 0$ if $r\in I_e$.  Now consider $I_e=[l,r]$.  When $m(\{l\}), m(\{r\})>0$,  clearly $\cH_\alpha\subset \cD(\sL)$.  When $m(\{l\})>0$ and $m(\{r\})=0$,  $h=c_1 u_++c_2 u_-\in  \cD(\sL)$ if and only if 
\[
c_1\frac{du_+}{d\bs}(r)+c_2\frac{du_-}{d\bs}(r)=0.  
\]
This amounts to $c_2/c_1=c_r$.  Hence $\cH_\alpha\cap \cD(\sL)=\text{span}\{u_++c_ru_-\}$.  The third case can be treated analogically.   When $m(\{l\})=m(\{r\})=0$,  $h=c_1 u_++c_2 u_-\in  \cD(\sL)$ if and only if 
\[
\begin{aligned}
&c_1\frac{du_+}{d\bs}(r)+c_2\frac{du_-}{d\bs}(r)=0, \\
&c_1\frac{du_+}{d\bs}(l)+c_2\frac{du_-}{d\bs}(l)=0.  
\end{aligned}\]
Note that $du/d\bs(l)<0,  du/d\bs(r)>0$ and $u(l), u(r)>0$.  
Then we have $c_1=c_2=0$ because 
\[
	\frac{du_+}{d\bs}(r)\frac{du_-}{d\bs}(l)-\frac{du_-}{d\bs}(r)\frac{du_+}{d\bs}(l)=-C^2\frac{du}{d\bs}(l)\frac{du}{d\bs}(r)-\frac{C}{u(r)}\frac{du}{d\bs}(l)+\frac{C}{u(l)}\frac{du}{d\bs}(r)>0.  
\]
That completes the proof. 
\end{proof}

\bibliographystyle{abbrv}
\bibliography{StructureOC}

\end{document}